\documentclass[11pt,reqno]{amsart}

\usepackage[utf8]{inputenc}

\usepackage{hyperref}
\usepackage{enumerate}
\usepackage{tikz-cd}   
\usepackage{graphicx}
\usepackage{tikz}
\usepackage{float}

\usepackage{longtable}

\usepackage{a4wide}
\newcommand{\n}{\overline{n}}

\newcommand{\Z}{\mathbb{Z}}

\DeclareMathOperator{\End}{End}

\DeclareMathOperator{\id}{id}

\title{Very good gradings on structural matrix rings}	

\author{Patrik Lundstr\"om}
\address{Department of Engineering Science,
University West, SE-46186 Trollh\"attan, Sweden}
\email{patrik.lundstrom@hv.se}

\author{Johan \"Oinert }
\address{Department of Mathematics and Natural Sciences,
Blekinge Institute of Technology,
SE-37179 Karlskrona, Sweden \ \ and \ \ Department of Engineering,
University of Sk\"{o}vde,
SE-54128 Sk\"{o}vde, Sweden}
\email{johan.oinert@bth.se}

\author{Laura Orozco}
\address{Departamento de Matem\'atica, Universidade Federal de Santa Catarina, 88040-970 Florian\'opolis SC, Brazil}
\email{lanaorga@gmail.com}
	
\author{H\'ector Pinedo}
\address{Escuela de Matematicas, Universidad Industrial de Santander, Cra. 27 Calle 9  UIS
	Edificio 45, Bucaramanga, Colombia} 
	\email{hpinedot@uis.edu.co}



\date{\today}

\usepackage{amsmath,amssymb,amsthm}

\theoremstyle{plain}
\newtheorem{theorem}{Theorem}
\newtheorem{lemma}[theorem]{Lemma}
\newtheorem{prop}[theorem]{Proposition}
\newtheorem{cor}[theorem]{Corollary}
\theoremstyle{definition}
\newtheorem{defn}[theorem]{Definition}

\newtheorem{exmp}[theorem]{Example}

\newtheorem{question}{Question}

\theoremstyle{remark}
\newtheorem{remark}[theorem]{Remark}

\begin{document}

\begin{abstract}
Let $R$ be a nonzero associative unital ring, let $G$ be a group,
and let $\rho$ be a preorder on $\{1,\ldots,n\}$. A $G$-grading on
$\rho$ induces a \emph{very good} $G$-grading on the structural
matrix ring $M_n(\rho,R)$. 
We show that, for each of the properties trivial, symmetric, epsilon-strong and strong, the grading on $\rho$ has the property if and only if the induced ring grading does.
The epsilon-crossed product and
crossed product properties pass from $\rho$ to the ring, but the
converses fail in general. We also give a concrete criterion for
epsilon-strongness and show that a very good $G$-grading on $M_n(\rho,R)$ that is strong satisfies
$|G|\leq n$. When $\rho$ is an equivalence relation and the neutral
component is diagonal, very good gradings correspond bijectively to
free partial actions of $G$ on $\{1,\ldots,n\}$ with orbit relation
$\rho$. These gradings are epsilon-crossed products, and over a field
the correspondence gives a classification up to graded algebra
isomorphism.
\end{abstract}

\maketitle

\section{Introduction}

Suppose that $R$ is a nonzero associative unital ring with multiplicative
identity $1$, and let $n$ be a positive integer. There are several
natural ways to construct subrings of $M_n(R)$, the ring of
$n\times n$ matrices over $R$. One such method is to prescribe the
\emph{shape} of the matrices, as in the subrings of diagonal, upper
triangular, or lower triangular matrices.

This idea was formalized by Mitchell in
\cite[p.~231]{mitchell1965}. Put
$\overline{n}:=\{1,\ldots,n\}$ and 
$\overline{n}^{\,2}:=\overline{n}\times\overline{n}$.
For each $(i,j)\in\overline{n}^{\,2}$, let $e_{i,j}$ denote the
matrix having $1$ in the $(i,j)$-entry and zeros 
elsewhere. Mitchell
defines a \emph{pattern} to be a subset $\rho$ of
$\overline{n}^{\,2}$. Associated with such a pattern is the 
additive subgroup $\sum_{(i,j)\in\rho}Re_{i,j}$
of $M_n(R)$, consisting of all matrices whose 
nonzero entries are
restricted to positions belonging to $\rho$.

This additive subgroup is closed under multiplication 
if and only if
$\rho$, viewed as a binary relation on $\overline{n}$, 
is transitive. It contains the identity matrix of $M_n(R)$ 
if and only if $\rho$ is reflexive. 
Consequently, it is a unital subring of $M_n(R)$ with 
the same identity if and only if $\rho$ is a preorder 
on $\overline{n}$.

For a preorder $\rho$ on $\overline{n}$, we write
$M_n(\rho,R):=\sum_{(i,j)\in\rho}Re_{i,j}$
for the corresponding subring of $M_n(R)$. Following van Wyk
\cite{vanwyck1988Maximal}, we call $M_n(\rho,R)$ a
\emph{structural matrix ring}. When $\rho$ is a partial 
order, that is, a reflexive, antisymmetric, and transitive
relation, the structural matrix ring $M_n(\rho,R)$ 
is the incidence algebra of the
finite poset $(\overline{n},\rho)$. Incidence algebras 
were introduced by Rota in \cite{rota1964}, 
independently of Mitchell's 
work \cite{mitchell1965}.

Structural matrix rings have been studied from several perspectives, including radical and ideal theory \cite{sands1990,vanwyck1988Maximal}, automorphism groups and isomorphism problems \cite{coelho1994,dascalescu1996}, and applications to coding theory \cite{kelarev2001}. 

In this article, we focus on a different aspect, namely gradings on structural matrix rings.
Let $G$ be a group with identity element $e$,
and let $S$ be an associative unital ring. Recall that $S$ is said to
be \emph{$G$-graded} if there is a family
$(S_g)_{g\in G}$ of additive subgroups of $S$ such that
$S=\bigoplus_{g\in G}S_g$ and $S_g S_h\subseteq S_{gh}$
for all $g,h\in G$. The elements of
$\bigcup_{g\in G}S_g$
are called \emph{homogeneous}. Group-graded rings include many
important classes of rings, such as polynomial rings, skew group
rings, twisted group rings, graded matrix rings, crossed products, and
their partial analogues; see, for example, \cite{nas04,NYO}.

Several important classes of group gradings have been introduced. The
following ones will be particularly relevant in this article. A
$G$-grading $(S_g)_{g\in G}$ on $S$ is said to be

\begin{itemize}

\item \emph{symmetric} if $S_gS_{g^{-1}}S_g=S_g$
for every $g\in G$;

\item \emph{strong} if $S_gS_{g^{-1}}=S_e$
for every $g\in G$;

\item a \emph{crossed product} if, for every $g\in G$, the 
component $S_g$ contains an element that is invertible in $S$;

\item \emph{epsilon-strong} if, for every $g\in G$, there is 
an element $\epsilon_g\in S_gS_{g^{-1}}$
such that
\begin{equation}\label{eqep}
\epsilon_gs=s\epsilon_{g^{-1}}=s
\end{equation}
for every $s\in S_g$;

\item an \emph{epsilon-crossed product} if it is 
epsilon-strong and,
for every $g\in G$, there are elements
$s_g \in S_g$ and $t_{g^{-1}}\in S_{g^{-1}}$ such that
$s_gt_{g^{-1}}=\epsilon_g$ and 
$t_{g^{-1}}s_g=\epsilon_{g^{-1}}$;

\item \emph{trivial} if $S_g = \{0\}$
for every $g\in G\setminus\{e\}$.

\end{itemize}

These classes satisfy certain containments illustrated by the following implications:
\[
\begin{array}{l}
\text{crossed product}
\Longrightarrow
\text{strong}
\Longrightarrow
\text{epsilon-strong}
\Longrightarrow
\text{symmetric},
\\[5pt]
\text{crossed product}
\Longrightarrow
\text{epsilon-crossed product}
\Longrightarrow
\text{epsilon-strong}.
\end{array}
\]
Moreover, every trivial grading is an epsilon-crossed product and is
therefore epsilon-strong and symmetric. In general, none of the
displayed implications can be reversed. For further details, see
\cite{nas04} for general group-graded rings and \cite{NYO} for epsilon-strongly
graded rings.

The starting point of this article is the following observation.
Suppose that $\rho$ is a preorder on $\overline{n}$ and that
$(\rho_g)_{g\in G}$ is a family of pairwise disjoint subsets of
$\rho$ such that
$\rho = \bigcup_{g\in G}\rho_g$ and 
$\rho_g\rho_h\subseteq\rho_{gh}$
for all $g,h\in G$. Then, for every $g\in G$, setting
$M_n(\rho,R)_g := \sum_{(i,j)\in\rho_g}Re_{i,j}$
defines a $G$-grading on $M_n(\rho,R)$; see
Proposition~\ref{prop:rhodefinesS}. In this situation, we say that
$\rho$ is $G$-graded and that the induced grading on
$M_n(\rho,R)$ is \emph{very good}.

This leads to the central question of this article.

\begin{question}
\emph{How are properties of a $G$-grading on $\rho$ related to the
corresponding properties of the induced very good $G$-grading on
$M_n(\rho,R)$? In particular, which grading properties pass between
the relation and the structural matrix ring?}
\end{question}

To the best of our knowledge, gradings on general structural matrix rings have previously been considered only in \cite{BD,BDV,DP}. Gradings on full matrix rings, that is, structural matrix rings corresponding to $\rho=\overline{n}^{\,2}$, have been studied much more extensively; see, for example, \cite{BSZ,BD,cae02,das99} and the references therein.
In that setting, the more general notion of a \emph{good} grading has
also been introduced: a grading is good if every matrix unit
$e_{i,j}$ is homogeneous. In this article, however, we focus
exclusively on very good gradings, namely those induced by
$G$-gradings on preorders as described above.

Here is an outline of this article. 
In
Section~\ref{gradedp}, we introduce $G$-gradings on preorders and
several important classes of such gradings. Their relationships are
described in Proposition~\ref{prop:implicationsrho}. We also prove in
Proposition~\ref{prop:i-ii-iii} that every graded equivalence relation
is epsilon-strong. As a consequence, Corollary~\ref{full} shows that
every $G$-grading on $\overline{n}^{\,2}$ is epsilon-strong.

In Section~\ref{structural}, we turn to structural matrix rings.
Proposition~\ref{prop:rhodefinesS} shows that every $G$-grading on a
preorder $\rho$ induces a very good $G$-grading on $M_n(\rho,R)$.
Theorem~\ref{relring} shows that triviality, symmetry,
epsilon-strongness, and strongness are preserved in both directions
when passing between a grading on $\rho$ and the induced ring
grading. It also shows that the epsilon-crossed product and
crossed product properties pass from $\rho$ to the ring. 
In Example~\ref{ex:crossed-converses} we show 
that the converses of these latter implications fail
for arbitrary coefficient rings. We then obtain a concrete
combinatorial criterion for epsilon-strongness and determine the
corresponding epsilon elements. 
We also prove that if $M_n(\rho,R)$ is equipped with a very good $G$-grading that is strong,
then $|G| \leq n$, with equality only when $\rho = \overline{n}^{\,2}$; see
Theorem~\ref{thm:strong-cardinality}.

Finally, in Section~\ref{partial}, we relate very good gradings to
partial group actions. When $\rho$ is an equivalence relation, we
establish a bijection between the very good $G$-gradings on
$M_n(\rho,R)$ whose neutral component consists precisely of the
diagonal matrices and the free partial actions of $G$ on
$\overline{n}$ whose orbit equivalence relation is $\rho$; see Theorem~\ref{thm:partial-action-correspondence}. 
Every
grading occurring in this correspondence is an epsilon-crossed
product, and strong gradings correspond precisely to global partial
actions. In the full matrix case, strictly elementary gradings
correspond to free and transitive partial actions; see Corollary~\ref{cor:strict-elementary-partial}. When the
coefficient ring is a field, we also classify the relevant gradings
up to graded algebra isomorphism; see Theorem~\ref{thm:classification-partial-actions} and Corollary~\ref{cor:classification-strict-elementary}.

\section{$G$-graded preorders on $\{1, \ldots, n\}$}\label{gradedp}

In this section, we introduce the notation and basic 
properties of $G$-graded preorders that will be used 
throughout this article. Let $\mathbb{N}$ denote the set of positive integers, and fix $n\in\mathbb{N}$. Put
$\overline{n} := \{1, \ldots, n\}$ and 
$\overline{n}^2 := \overline{n} \times \overline{n}$.
We define an inversion operation on $\overline{n}^2$ by
\[ (i,j)^{-1} := (j,i), \]
for all $i, j \in \overline{n}$, and a partial
multiplication by
\[ (i,j)(j,k) := (i,k), \]
for all $i,j,k \in \overline{n}$.
More generally, for subsets $X, Y \subseteq \overline{n}^2$, we define
\[ 
X^{-1} := \left\{ (j,i) \in \overline{n}^2 \mid (i,j) \in X \right\}
\] 
and
\begin{displaymath}
XY := \left\{ (i,k) \in \overline{n}^2 \mid \text{there exists } j \in \overline{n} \text{ such that } (i,j) \in X, (j,k) \in Y \right\}.
\end{displaymath}

We denote by $\Delta := \{ (i,i) \mid i \in \overline{n} \}$ the \emph{diagonal} of $\overline{n}^2$.
Let $\rho \subseteq \overline{n}^2$ be a nonempty binary relation on $\overline{n}$. Recall that $\rho$ is called \emph{reflexive} if $\Delta \subseteq \rho$, and \emph{transitive} if $\rho \rho \subseteq \rho$. 

Throughout this section, $\rho$ denotes a preorder on 
$\overline{n}$, and $G$ denotes a group with identity 
element $e$. 

\begin{defn}\label{gradedr}
We say that $\rho$ is \emph{$G$-graded} if there is a 
family  $(\rho_g)_{g\in G}$ of pairwise disjoint subsets of
$\rho$ such that
$\rho = \bigcup_{g\in G}\rho_g$ and 
$\rho_g\rho_h\subseteq\rho_{gh}$
for all $g,h\in G$.
In that case, we will call 
$(\rho_g)_{g \in G}$ a \emph{$G$-grading on $\rho$}.
\end{defn}

The following elementary observation will be used repeatedly
later on.

\begin{prop}\label{prop:deltasubsetrho}
Suppose that $\rho$ is $G$-graded. Then the following assertions hold:
\begin{enumerate}[{\rm (i)}]

\item $\Delta\subseteq\rho_e$;

\item for every $g\in G$,
$\rho_g = \Delta\rho_g = \rho_g\Delta =
\rho_e\rho_g = \rho_g\rho_e$.

\end{enumerate}
\end{prop}

\begin{proof}
For (i), let $i\in\overline{n}$. 
Then $(i,i)\in\rho_g$ for some
$g\in G$. Since $(i,i) = (i,i)(i,i) \in
\rho_g\rho_g \subseteq \rho_{g^2}$, we have
$(i,i)\in\rho_g\cap\rho_{g^2}$.
The sets $\rho_h$, $h\in G$, are pairwise disjoint, and hence
$g=g^2$. Therefore, $g=e$, which proves that
$\Delta\subseteq\rho_e$.

For (ii), let $g\in G$. Since $\rho$ is reflexive, we have
$\Delta\rho_g=\rho_g=\rho_g\Delta$.
By (i) and the grading property,
$\rho_g = \Delta\rho_g \subseteq \rho_e\rho_g
\subseteq \rho_g$.
Thus $\rho_e\rho_g=\rho_g$. Similarly,
$\rho_g = \rho_g\Delta \subseteq \rho_g\rho_e
\subseteq \rho_g$,
and hence $\rho_g\rho_e=\rho_g$.
\end{proof}

\begin{defn}\label{gradedr2}
We say that a $G$-grading $(\rho_g)_{g \in G}$ on the preorder $\rho$ is:
\begin{itemize}

\item \emph{trivial} if for each 
$g \in G \setminus \{ e \}$, the equality
$\rho_g = \emptyset$ holds;

\item \emph{symmetric} if for each $g \in G$
the equality
$\rho_g \rho_{g^{-1}} \rho_g = \rho_g$ holds;

\item \emph{strong} if for each $g \in G$
the equality
$\rho_g \rho_{g^{-1}} = \rho_e$ holds;

\item a \emph{crossed product} if for each 
$g \in G$ there are $U_g \subseteq \rho_g$ and 
$V_{g^{-1}} \subseteq \rho_{g^{-1}}$ such that 
the equalities
$U_g V_{g^{-1}} = V_{g^{-1}} U_g = \Delta$ hold;

\item \emph{epsilon-strong} if 
for each $g\in G$ there is a 
subset $E_g$ of $(\rho_g \rho_{g^{-1}}) \cap 
\Delta$ such that for every $(i,j) \in \rho_g$, 
the equalities 
$(i,j) = E_g (i,j) = (i,j) E_{g^{-1}}$ hold;

\item an \emph{epsilon-crossed product} if 
$\rho$ is epsilon-strong (with the sets
$E_g$ as above) and for each 
$g\in G$, there are $U_g \subseteq \rho_g$ and 
$V_{g^{-1}} \subseteq \rho_{g^{-1}}$ such that 
$U_g V_{g^{-1}} = E_g$ and  
$V_{g^{-1}} U_g = E_{g^{-1}}$.

\end{itemize}
\end{defn}

\begin{remark}\label{rem:epsilon-set}
Suppose that $\rho$ is epsilon-strongly $G$-graded. Then, for every
$g\in G$, the epsilon set $E_g$ is uniquely determined and is given by
$E_g=\rho_g\rho_{g^{-1}}\cap\Delta$.
Indeed, the inclusion
$E_g\subseteq\rho_g\rho_{g^{-1}}
\cap\Delta$
holds by definition. Conversely, let
$(i,i)\in\rho_g\rho_{g^{-1}}\cap\Delta$.
Then there is some $j\in\overline{n}$ such that
$(i,j)\in\rho_g$ and 
$(j,i)\in\rho_{g^{-1}}$.
Since $E_g(i,j)=(i,j)$, 
it follows that
$(i,i)\in E_g$. Thus
$\rho_g\rho_{g^{-1}}\cap\Delta
\subseteq E_g$.
\end{remark}

\begin{prop}
Suppose that $\rho$ is $G$-graded. The following
assertions are equivalent:
\begin{enumerate}[{\rm (i)}]

\item $\rho$ is strongly $G$-graded; 

\item for all $g,h \in G$, the equality 
$\rho_g \rho_h = \rho_{gh}$ holds.

\end{enumerate}
\end{prop}

\begin{proof}
(i)$\Rightarrow$(ii):
Suppose that $\rho$ is strongly $G$-graded. 
Take $g,h \in G$. Then
$\rho_g \rho_h \subseteq \rho_{gh} = \rho_{gh} 
\rho_e
= \rho_{gh} \rho_{h^{-1}} \rho_h \subseteq 
\rho_{gh h^{-1}} \rho_h = \rho_g \rho_h$.
Therefore, $\rho_g \rho_h = \rho_{gh}$.

(ii)$\Rightarrow$(i): This is trivial.
\end{proof}

\begin{defn}
Suppose that $\rho$ is $G$-graded.  
A subset $\mathcal{I} \subseteq \rho_e$ is called an \emph{ideal} of $\rho_e$ if 
$\mathcal{I} \rho_e \subseteq \mathcal{I}$ and 
$\rho_e \mathcal{I} \subseteq \mathcal{I}$.  
We say that the ideal $\mathcal{I}$ is \emph{unital} if there is a subset $I \subseteq \mathcal{I} \cap \Delta$ such that  
$I(i,j) = (i,j) = (i,j)I$ for all 
$(i,j) \in \mathcal{I}$ . In that case we will refer to $I$ as 
the \emph{unit} of $\mathcal{I}$.
\end{defn}

\begin{prop}\label{prop:epsilon-characterization}
Suppose that $\rho$ is $G$-graded. The following
assertions are equivalent:
\begin{enumerate}[{\rm (i)}]

\item $\rho$ is epsilon-strongly $G$-graded;

\item for each $g \in G$, $\rho_g \rho_{g^{-1}}$
is a unital ideal of $\rho_e$ such that for all
$g,h \in G$, the equalities $\rho_g \rho_h =
\rho_g \rho_{g^{-1}} \rho_{gh} = 
\rho_{gh} \rho_{h^{-1}} \rho_h$ hold;

\item $\rho$ is symmetrically $G$-graded and
for each $g \in G$, $\rho_g \rho_{g^{-1}}$ is 
a unital ideal of $\rho_e$.

\end{enumerate}
\end{prop}

\begin{proof}
(i)$\Rightarrow$(ii):
Suppose that $\rho$ is epsilon-strongly $G$-graded.
Take $g \in G$.
Clearly $\rho_g \rho_{g^{-1}}$ is an ideal of $\rho_e$. 
By Remark~\ref{rem:epsilon-set}, $E_g=
\rho_g \rho_{g^{-1}} \cap \Delta$.
We claim that $E_g$ is a unit for  $\rho_g \rho_{g^{-1}}$.
Indeed, take $(i,j) \in \rho_g \rho_{g^{-1}}$.
There is some $k \in \overline{n}$ such that 
$(i,k) \in \rho_g$ and $(k,j) \in \rho_{g^{-1}}$.
Note that $E_g(i,k) = (i,k)$
and $(k,j)E_g = (k,j)$. Thus,
$E_g(i,j) = E_g (i,k)(k,j) = (i,k)(k,j) = (i,j)$ and
$(i,j) E_g = (i,k)(k,j) E_g = (i,k)(k,j) = (i,j),$ as desired.
Now, take $g,h \in G$. Then
\[
\rho_g \rho_h = E_g \rho_g \rho_h \subseteq 
\rho_g \rho_{g^{-1}} \rho_g \rho_h \subseteq 
\rho_g \rho_{g^{-1}} \rho_{gh} \subseteq 
\rho_g \rho_{g^{-1} g h} = \rho_g \rho_h
\]
and
\[
\rho_g \rho_h = \rho_g \rho_h E_{h^{-1}} \subseteq 
\rho_g \rho_h \rho_{h^{-1}} \rho_h \subseteq 
\rho_{gh} \rho_{h^{-1}} \rho_h \subseteq 
\rho_{g h h^{-1}} \rho_h  = \rho_g \rho_h
\]
showing that  $\rho_g \rho_h =
\rho_g \rho_{g^{-1}} \rho_{gh} = 
\rho_{gh} \rho_{h^{-1}} \rho_h$.

(ii)$\Rightarrow$(iii):
Put $h = e$ in (ii).

(iii)$\Rightarrow$(i):
Suppose that (iii) holds. For each $g \in G$,  
let $E_g$ denote the unit of the ideal 
$\rho_g \rho_{g^{-1}}$. Note that $E_g \subseteq \Delta$.
Take $g \in G$
and $(i,j) \in \rho_g$. Since $\rho$ is 
symmetrically $G$-graded, there are 
$k,l \in \overline{n}$ such that 
$(i,k),(l,j) \in \rho_g$ and $(k,l) \in \rho_{g^{-1}}$.
From $(i,k)(k,l) \in \rho_g \rho_{g^{-1}}$ and
$(k,l)(l,j) \in \rho_{g^{-1}} \rho_g$ it follows that
$E_g (i,k)(k,l) = (i,k)(k,l)$ and
$(k,l)(l,j) E_{g^{-1}} = (k,l)(l,j)$. Hence,
$E_g (i,j) = E_g (i,k)(k,l)(l,j) = (i,k)(k,l)(l,j)=(i,j)$
and
$(i,j) E_{g^{-1}} = (i,k)(k,l)(l,j) E_{g^{-1}} =
(i,k)(k,l)(l,j) = (i,j)$,
showing that $\rho$ is epsilon-strongly $G$-graded.
\end{proof}

\begin{prop}\label{prop:implicationsrho}
Suppose that $\rho$ is $G$-graded. 
Consider the following properties:
\begin{enumerate}[{\rm (i)}]

\item $\rho$ is symmetrically graded;

\item $\rho$ is epsilon-strongly graded;

\item $\rho$ is an epsilon-crossed product;

\item $\rho$ is strongly graded;

\item $\rho$ is a crossed product;

\item $\rho$ is trivially graded.

\end{enumerate}
These properties are related by the following implications:
\[
(v) \Rightarrow (iii) \Rightarrow (ii) \Rightarrow (i), \quad  
(v) \Rightarrow (iv) \Rightarrow (ii), \quad
\mbox{and} \quad (vi) \Rightarrow (iii).
\]
\end{prop}

\begin{proof}
(v)$\Rightarrow$(iii): 
Let $\rho$ be a
crossed product with $U_g \subseteq \rho_g$
and $V_{g^{-1}} \subseteq \rho_{g^{-1}}$ such that
$U_g V_{g^{-1}} = V_{g^{-1}} U_g = \Delta$, 
for $g \in G$. Then $\rho$ is an
epsilon-crossed product with $E_g := \Delta$,
for $g \in G$.

(iii)$\Rightarrow$(ii): 
This is trivial.

(ii)$\Rightarrow$(i):
This follows from Proposition~\ref{prop:epsilon-characterization}.

(v)$\Rightarrow$(iv): 
Suppose that $\rho$ is a 
crossed product. Take $g \in G$. 
Since $\rho_e \Delta = \rho_e$, it follows that
$\rho_e = \rho_e \Delta = 
\rho_e U_g V_{g^{-1}} \subseteq \rho_e \rho_g 
\rho_{g^{-1}} \subseteq \rho_{e g g^{-1}} = \rho_e$.
Therefore, $\rho_e = \rho_e \rho_g \rho_{g^{-1}} = \rho_g \rho_{g^{-1}}$,
showing that $\rho$ is strongly $G$-graded.

(iv)$\Rightarrow$(ii): 
Suppose that $\rho$ is strongly graded.
It is easy to see that $\rho$ is 
epsilon-strongly graded with $E_g := \Delta,$
for all $g \in G$.

(vi)$\Rightarrow$(iii):
Suppose that $\rho$ is trivially graded.
Then $\rho$ is an epsilon-crossed product with
$U_e = V_e = E_e = \Delta$ and 
$U_g = V_g = E_g = \emptyset$, for 
$g \in G \setminus \{ e \}$.
\end{proof}

\begin{exmp}
The following examples show that none of the implications in
Proposition~\ref{prop:implicationsrho} can be reversed in general.

(a) (iii)$\not\Rightarrow$(v).
Let $\rho:=\overline{1}^{\,2}=\{(1,1)\}$, and equip $\rho$ 
with the trivial $\Z_2$-grading
$\rho_0 := \rho$ and $\rho_1 := \emptyset$.
By Proposition~\ref{prop:implicationsrho}, this grading is an
epsilon-crossed product. It is not a crossed product, 
since there are no subsets
$U_1 , V_1 \subseteq \rho_1 = \emptyset$ with
$U_1 V_1 = \Delta$.

(b) (ii)$\not\Rightarrow$(iii) and
(ii)$\not\Rightarrow$(iv).
Let $\rho$ be the preorder on $\overline{4}$ defined by
$\rho:=\overline{3}^{\,2}\cup\{(4,4)\}$.
Define a $\Z_2$-grading on $\rho$ by
\[ 
\rho_0 := \{(1,1),(2,2),(2,3),(3,2),(3,3),(4,4)\}
\quad \mbox{and} \quad \rho_1 := \{(1,2),(1,3),(2,1),(3,1)\}.
\]
This grading is epsilon-strong with
$E_0 := \Delta$ and $E_1 := \{(1,1),(2,2),(3,3)\}$.
However, $\rho_1 \rho_1 = \rho_0\setminus\{(4,4)\}
\neq \rho_0$, so the grading is not strong.
We claim that it is not an epsilon-crossed product. 
Indeed, by Remark~\ref{rem:epsilon-set},
the
epsilon set in degree $1$ must be
$E_1=\rho_1\rho_1\cap\Delta
=\{(1,1),(2,2),(3,3)\}$.
Suppose that there are subsets
$U_1,V_1\subseteq\rho_1$ such that $U_1 V_1 = E_1$.
Since $(2,2)\in U_1V_1$, we must have
$(2,1) \in U_1$ and $(1,2) \in V_1$.
Similarly, since $(3,3)\in U_1V_1$, we must have
$(3,1) \in U_1$ and $(1,3) \in V_1$. It follows that
$(2,3) = (2,1)(1,3)\in U_1 V_1$, contrary to the 
equality $U_1 V_1 = E_1$. Thus the grading is not an
epsilon-crossed product.

(c) (i)$\not\Rightarrow$(ii).
Let $\rho$ be the preorder on $\overline{4}$ defined by
\[
\Delta \cup
\{ (1,2),(1,3),(1,4),(2,3),(2,4), (3,2),(3,4),(4,2),(4,3) \}.
\]
Define a $\Z_3$-grading on $\rho$ by
\[
\rho_0 := \Delta\cup\{(1,3),(2,4),(4,2)\},
\
\rho_1 := \{(1,2),(1,4),(3,2),(3,4)\}
\ \mbox{and} \ 
\rho_2 := \{(2,3),(4,3) \}.
\]
A direct computation gives
$\rho_0\rho_0\rho_0=\rho_0$,
$\rho_1\rho_2\rho_1=\rho_1$ and
$\rho_2\rho_1\rho_2=\rho_2$.
Hence the grading is symmetric.
It is not epsilon-strong. Indeed,
$\rho_1 \rho_2 = \{(1,3),(3,3)\}$.
If there were a subset
$E_1 \subseteq (\rho_1\rho_2) \cap \Delta$
such that $E_1 (1,2) = (1,2)$, then $(1,1)$ would 
belong to $E_1$. This is impossible, since
$(1,1) \notin \rho_1\rho_2$.

(d) (iv)$\not\Rightarrow$(v).
Let $\rho:=\overline{3}^{\,2}$, and define a $\Z_2$-grading on $\rho$ by
\[
\rho_0 := \{(1,1),(2,2),(2,3),(3,2),(3,3)\}
\quad \mbox{and} \quad
\rho_1 := \{(1,2),(1,3),(2,1),(3,1)\}.
\]
Since $\rho_0 \rho_0 = \rho_0$ and $\rho_1\rho_1 = \rho_0$,
the grading is strong.
Suppose, seeking a contradiction, that the grading is a 
crossed product. Then there are subsets
$U_1 , V_1 \subseteq \rho_1$ such that
$U_1 V_1 = V_1 U_1 = \Delta$.
Since $(2,2)$ belongs to both products, it follows that
$(2,1),(1,2)\in U_1\cap V_1$.
Similarly, since $(3,3)$ belongs to both products, we obtain
$(3,1),(1,3) \in U_1 \cap V_1$. Therefore,
$U_1 = V_1 = \rho_1$. Consequently,
$U_1 V_1 = \rho_1 \rho_1 = \rho_0 \neq \Delta$,
which is a contradiction. Hence the grading is 
not a crossed product.

(e) (iii)$\not\Rightarrow$(vi).
Let $\rho:=\overline{2}^{\,2}$, and define a $\Z_2$-grading on $\rho$ by $\rho_0: = \Delta$ and 
$\rho_1:=\{(1,2),(2,1)\}$. Put $U_0 = V_0: =\Delta$
and $U_1 = V_1: = \rho_1$. Then
$U_0 V_0 = V_0 U_0 = \Delta$ and
$U_1 V_1 = V_1 U_1 = \Delta$.
Thus the grading is a crossed product and hence 
an epsilon-crossed
product. Since $\rho_1\neq\emptyset$, it is not trivial.
\end{exmp}

Recall that a binary relation $\rho$ on $\overline{n}$ 
is called \emph{symmetric} if $\rho^{-1}=\rho$. 
Since $\rho$ is already assumed to be a preorder, it is 
an equivalence relation if and only if it is symmetric.

\begin{lemma}\label{lem:TransposeInverse}
Suppose that $\rho$ is $G$-graded.
If $(i,j) \in \rho_g$ and $(j,i) \in \rho$, then $(j,i) \in \rho_{g^{-1}}$.
\end{lemma}

\begin{proof}
Suppose that $(i,j) \in \rho_g$ and $(j,i) \in \rho$.
Then $(j,i) \in \rho_h$ for some $h\in G$.
Note that $\rho_e \supseteq \Delta \ni (i,i) = (i,j) (j,i) \in \rho_g \rho_h \subseteq \rho_{gh}$.
Thus, $e=gh$, showing that $h=g^{-1}$. 
\end{proof}

\begin{prop}\label{prop:i-ii-iii}
Suppose that $\rho$ is $G$-graded.
The following assertions are equivalent:
\begin{enumerate}[{\rm (i)}]

\item $\rho$ is an equivalence relation;

\item $\rho_e$ is an equivalence relation
and $\rho$ is epsilon-strongly $G$-graded;

\item $\rho_e$ is an equivalence relation
and $\rho$ is symmetrically $G$-graded;

\item for each $g \in G$, the equality
$(\rho_g)^{-1} = \rho_{g^{-1}}$ holds.

\end{enumerate}
\end{prop}

\begin{proof}
(i)$\Rightarrow$(ii):
Suppose that (i) holds. 
It is easy to see that $\rho_e$ is reflexive and transitive.
The symmetry of $\rho_e$ follows from Lemma~\ref{lem:TransposeInverse}.
Now, take $g \in G$ and $(i,j) \in \rho_g$.
By Lemma~\ref{lem:TransposeInverse}, $(j,i) \in \rho_{g^{-1}}$.
Set $E_g:=\rho_g\rho_{g^{-1}}\cap \Delta.$
Note that $(i,i)=(i,j)(j,i) \in E_g$ and $(j,j)=(j,i)(i,j) \in E_{g^{-1}}$.
Clearly, $E_g (i,j) = (i,j) = (i,j)E_{g^{-1}}$.
This shows that $\rho$ is epsilon-strongly $G$-graded.

(ii)$\Rightarrow$(iii): This follows from Proposition~\ref{prop:implicationsrho}.

(iii)$\Rightarrow$(iv):
Suppose that (iii) holds. Let $g\in G$ and 
$(i,j)\in\rho_g$. Since the grading is symmetric, 
$(i,j) \in \rho_g \rho_{g^{-1}} \rho_g$. 
Hence there are $k,l\in\overline{n}$ with 
$(i,k) \in \rho_g$, $(k,l) \in \rho_{g^{-1}}$ and
$(l,j) \in \rho_g$. 
The grading property gives 
$(i,l) = (i,k) (k,l) \in \rho_e$ and
$(k,j) = (k,l)(l,j) \in \rho_e$.
Since $\rho_e$ is symmetric, 
$(l,i)\in\rho_e$ and $(j,k) \in \rho_e$. 
Therefore, $(j,l) = (j,k)(k,l) \in \rho_{g^{-1}}$,
and hence $(j,i) = (j,l)(l,i) \in \rho_{g^{-1}}\rho_e 
= \rho_{g^{-1}}$, where the last equality follows from
Proposition~\ref{prop:deltasubsetrho}(ii). 
Thus  $(\rho_g)^{-1} \subseteq \rho_{g^{-1}}$.
Replacing $g$ by $g^{-1}$ gives
$(\rho_{g^{-1}})^{-1} \subseteq \rho_g$.
Taking inverses yields 
$\rho_{g^{-1}}\subseteq(\rho_g)^{-1}$.
Consequently, $(\rho_g)^{-1} = \rho_{g^{-1}}$.

(iv)$\Rightarrow$(i):
Suppose that (iv) holds. 
Then $\rho^{-1} = 
\left( \cup_{g \in G} \rho_g \right)^{-1} = 
\cup_{g \in G} \left( \rho_g \right)^{-1} = 
\cup_{g \in G} \rho_{g^{-1}} = \rho$
so that $\rho$ is symmetric, and hence an equivalence relation.
\end{proof}

\begin{cor}\label{full}
Every $G$-grading on $\overline{n}^{\,2}$ is epsilon-strong.
\end{cor}

\begin{proof}
Since $\overline{n}^{\,2}$ is an equivalence relation, the result
follows from Proposition~\ref{prop:i-ii-iii}.
\end{proof}

The epsilon-strong gradings arising from equivalence 
relations need not be trivial, as the following example shows.

\begin{exmp}
There exist epsilon-strong gradings on equivalence relations that are
not trivial. Indeed, let $n:=4$ and consider the equivalence relation
\[
\rho
:=
\{(i,j)\mid 1\leq i,j\leq 2\}
\cup
\{(i,j)\mid 3\leq i,j\leq 4\}.
\]
Thus the equivalence classes of $\rho$ are $\{1,2\}$ and
$\{3,4\}$. Let $G := \langle g\mid g^2=e\rangle$
be the cyclic group of order $2$, and define a $G$-grading on $\rho$ by $\rho_e := \{(1,1),(2,2),(3,3),(4,4)\}$
and $\rho_g := \{(1,2),(2,1),(3,4),(4,3)\}$.
Since $\rho$ is an equivalence relation,
Proposition~\ref{prop:i-ii-iii} shows that the grading is
epsilon-strong. However, it is not trivial, since
$\rho_g \neq \emptyset$.
\end{exmp}

The situation is entirely different when the underlying
preorder is antisymmetric.

\begin{prop}\label{esgo}
Every epsilon-strong $G$-grading on a partial order on
$\overline{n}$ is trivial.
\end{prop}

\begin{proof}
Let $\rho$ be a partial order on $\overline{n}$ equipped 
with an epsilon-strong $G$-grading, and let $g\in G$. 
Suppose that
$\rho_g\neq\emptyset$, and choose $(i,j)\in\rho_g$.
By epsilon-strongness and Remark~\ref{rem:epsilon-set}, 
$E_g=
\bigl(\rho_g\rho_{g^{-1}}\bigr)\cap\Delta$
with
$E_g (i,j) = (i,j)$. 
Hence,
$(i,i) \in E_g \subseteq \rho_g \rho_{g^{-1}}$, 
so there is some $k\in\overline{n}$ such that
$(i,k) \in \rho_g$ and $(k,i)\in\rho_{g^{-1}}$.
In particular, $(i,k),(k,i)\in\rho$. Since $\rho$ 
is antisymmetric, we obtain $i=k$. Thus
$(i,i) \in \rho_g$. On the other hand,
Proposition~\ref{prop:deltasubsetrho}(i) gives
$(i,i)\in\Delta\subseteq\rho_e$.
Since the homogeneous components are pairwise disjoint, 
it follows that $g=e$.
Thus, $\rho_g=\emptyset$ for $g\in G\setminus\{e\}$, and 
so the grading is trivial.
\end{proof}

\section{Graded structural matrix rings}\label{structural}

Throughout this section, we let $R$ be a nonzero 
associative unital ring,
and we denote its multiplicative identity by $1$. 
Let $M_n(R)$ denote 
the ring of $n\times n$ matrices over $R$. For each 
$(i,j)\in\overline{n}^{\,2}$, let $e_{i,j}$ denote the 
matrix having $1$ in the $(i,j)$-entry and zeros elsewhere. 
As in the preceding section, $\rho$ denotes a preorder 
on $\overline{n}$.

\begin{defn}
The \emph{structural matrix ring}
$M_n(\rho,R)$ induced by the relation $\rho$
is defined to be the additive subgroup 
$\sum_{(i,j) \in \rho} R e_{i,j}$ of 
$M_n(R)$ equipped with ordinary matrix
multiplication.
\end{defn}

Since $\rho$ is reflexive, 
$1_\Delta := \sum_{i=1}^n e_{i,i}$
belongs to $M_n(\rho,R)$ and is its multiplicative identity.

\begin{prop}\label{prop:rhodefinesS}
Suppose that $\rho$ is $G$-graded. If we put
$M_n(\rho,R)_g := \sum_{(i,j) \in \rho_g} R e_{i,j}$
for all $g \in G$, then this defines a
$G$-grading on $M_n(\rho,R)$.
\end{prop}

\begin{proof}
Put $S := M_n(\rho,R)$, and $S_g := M_n(\rho,R)_g$
for all $g \in G$.
Since the subsets $\rho_g$, $g\in G$, are pairwise 
disjoint and their union is $\rho$, we have 
$S = \bigoplus_{g\in G}S_g$ as additive groups.
Take $g,h \in G$. Then
\[ 
\begin{array}{rcl}
S_g S_h &=& 
\left( \sum_{(i,j) \in \rho_g} R e_{i,j} \right)
\left( \sum_{(i',j') \in \rho_h} R e_{i',j'} \right) 
= \sum_{(i,j) \in \rho_g} \sum_{(i',j') \in \rho_h} 
R e_{i,j} Re_{i',j'} \\
&=&  \sum_{(i'',j'') \in \rho_g \rho_h} R e_{i'',j''}  
\ \subseteq \ \sum_{(i'',j'') \in \rho_{gh}} 
R e_{i'',j''} \ = \ S_{gh}. 
\end{array} 
\]
This shows that $S$ is a $G$-graded ring.
\end{proof}

\begin{defn}\label{def:VeryGood}
A $G$-grading on $M_n(\rho,R)$ is called \emph{very good} if it is
induced by a $G$-grading on $\rho$ as in
Proposition~\ref{prop:rhodefinesS}.
\end{defn}

\begin{remark}\label{gradedrho}
(a)
Suppose that $M_n(\rho,R)$ is equipped with a very good $G$-grading.
Note that, for any $g\in G$, we have $(i,j)\in \rho_g$ if and only if $Re_{i,j}\subseteq M_n(\rho,R)_g.$

(b)
Definition~\ref{def:VeryGood} generalizes \cite[Def.~2.1(ii)]{LOOP}.
Indeed, if $\rho = \overline{n}^2$,
then a very good $G$-grading on 
$M_n(\rho,R) = M_n(R)$, in the sense of Definition~\ref{def:VeryGood}, is exactly a 
very good $G$-grading on $M_n(R)$, in the sense 
of \cite[Def.~2.1(ii)]{LOOP}.
\end{remark}

For the remainder of this section, we assume that 
$M_n(\rho,R)$ is
equipped with a very good $G$-grading induced by a $G$-grading
$(\rho_g)_{g\in G}$ on $\rho$.

\begin{remark}
Note that $1_\Delta \in M_n(\rho,R)_e$. Indeed, 
Proposition~\ref{prop:deltasubsetrho}(i) gives
$\Delta\subseteq\rho_e$, and therefore
$1_\Delta=\sum_{i=1}^n e_{i,i}\in M_n(\rho,R)_e$. See also \cite[Prop.~1.1.1]{nas04}.
\end{remark}

\begin{theorem}\label{relring}
The following assertions hold:
\begin{enumerate}[{\rm (i)}]

\item the $G$-grading on $\rho$ is trivial if and only if the induced
$G$-grading on $M_n(\rho,R)$ is trivial;

\item the $G$-grading on $\rho$ is symmetric if and only if the
induced $G$-grading on $M_n(\rho,R)$ is symmetric;

\item the $G$-grading on $\rho$ is epsilon-strong if and only if the
induced $G$-grading on $M_n(\rho,R)$ is epsilon-strong;

\item the $G$-grading on $\rho$ is strong if and only if the induced
$G$-grading on $M_n(\rho,R)$ is strong;

\item if the $G$-grading on $\rho$ is an epsilon-crossed product,
then the induced $G$-grading on $M_n(\rho,R)$ is an
epsilon-crossed product;

\item if the $G$-grading on $\rho$ is a crossed product, then the
induced $G$-grading on $M_n(\rho,R)$ is a crossed product.

\end{enumerate}
\end{theorem}

\begin{proof}
Put $S:=M_n(\rho,R)$. For all $g,h\in G$, we have
\begin{equation}\label{eq:component-products}
S_gS_h = \bigoplus_{(i,j)\in\rho_g\rho_h}Re_{i,j}.
\end{equation}
Indeed, the inclusion from left to right follows from 
ordinary matrix multiplication. Conversely, if
$(i,k)\in\rho_g\rho_h$,
then there is some $j\in\overline{n}$ such that
$(i,j)\in\rho_g$ and $(j,k)\in\rho_h$.
Hence, for every $r\in R$,
$re_{i,k}=(re_{i,j})e_{j,k}\in S_gS_h$.
Since $R\neq\{0\}$, equality of two subgroups of the form
$\bigoplus_{(i,j)\in X}Re_{i,j}$
is equivalent to equality of their indexing sets. It follows
immediately that the grading on $\rho$ is trivial 
if and only if the
grading on $S$ is trivial. Similarly, applying
\eqref{eq:component-products} twice shows that
$S_gS_{g^{-1}}S_g =
\bigoplus_{(i,j)\in
\rho_g\rho_{g^{-1}}\rho_g}Re_{i,j}$.
Consequently, the grading on $\rho$ is symmetric 
if and only if the
grading on $S$ is symmetric. Equation
\eqref{eq:component-products} also shows that the grading on $\rho$ is strong if and only if the grading on $S$ is strong.

Next we consider epsilon-strongness. 
Suppose first that the grading on $\rho$ is epsilon-strong.
By Remark~\ref{rem:epsilon-set},
$E_g=\rho_g\rho_{g^{-1}}\cap\Delta$.
Put
$\epsilon_g:=\sum_{(i,i)\in E_g}e_{i,i}$. 
By \eqref{eq:component-products},
$\epsilon_g\in S_gS_{g^{-1}}$. Moreover,
$\epsilon_gs=s=s\epsilon_{g^{-1}}$
for every $s\in S_g$. Thus the grading on $S$ 
is epsilon-strong. 

Conversely, suppose that the grading on $S$ is epsilon-strong. For
every $g\in G$, choose $\epsilon_g\in S_gS_{g^{-1}}$
such that $\epsilon_gs=s=s\epsilon_{g^{-1}}$
for every $s\in S_g$. Write
$\epsilon_g = \sum_{(k,l)\in\rho_g\rho_{g^{-1}}}
r_{k,l}^{(g)}e_{k,l}$,
where all but finitely many coefficients are zero, and define
$E_g := \bigl\{ (i,i)\in\Delta \bigm| r_{i,i}^{(g)}=1
\bigr\}$.
Since $\epsilon_g\in S_gS_{g^{-1}}$, we have
$E_g\subseteq \bigl(\rho_g\rho_{g^{-1}}\bigr)\cap\Delta$.
Let $(i,j)\in\rho_g$. From $\epsilon_ge_{i,j}=e_{i,j}$
it follows that $r_{i,i}^{(g)}=1$,
and hence $(i,i)\in E_g$. Similarly, from
$e_{i,j}\epsilon_{g^{-1}}=e_{i,j}$
it follows that $(j,j)\in E_{g^{-1}}$. Therefore,
$E_g(i,j)=(i,j)=(i,j)E_{g^{-1}}$.
Thus the grading on $\rho$ is epsilon-strong.

Now suppose that the grading on $\rho$ is an epsilon-crossed
product. For every $g\in G$, choose subsets
$U_g\subseteq\rho_g$ and 
$V_{g^{-1}}\subseteq\rho_{g^{-1}}$ such that
$U_gV_{g^{-1}}=E_g$ and $V_{g^{-1}}U_g=E_{g^{-1}}$.
Put
\[
\widetilde U_g := \{(i,j)\in U_g\mid(j,i)\in V_{g^{-1}}\}
\quad \mbox{and} \quad
\widetilde V_{g^{-1}} :=
\{(j,i)\mid(i,j)\in\widetilde U_g\}.
\]
The two displayed relation equalities imply that
$\widetilde U_g$ is the graph of a bijection between the diagonal
indices occurring in $E_{g^{-1}}$ and those occurring in $E_g$.
Consequently, if
\[
s_g
:=
\sum_{(i,j)\in\widetilde U_g}e_{i,j}
\qquad\text{and}\qquad
t_{g^{-1}}
:=
\sum_{(j,i)\in\widetilde V_{g^{-1}}}e_{j,i},
\]
then $s_gt_{g^{-1}} = \sum_{(i,i)\in E_g}e_{i,i}= \epsilon_g$
and $t_{g^{-1}}s_g = \sum_{(j,j)\in E_{g^{-1}}}e_{j,j}
= \epsilon_{g^{-1}}$.
Hence the induced grading on $S$ is an epsilon-crossed product.

Finally, suppose that the grading on $\rho$ is a crossed product.
The same argument applies with $E_g=E_{g^{-1}}=\Delta$.
It yields elements $u_g\in S_g$ and 
$v_{g^{-1}}\in S_{g^{-1}}$ such that
$u_gv_{g^{-1}}=v_{g^{-1}}u_g=1_\Delta$.
Thus $u_g$ is invertible in $S$, with inverse 
$v_{g^{-1}}$, and the induced grading is a crossed product.
\end{proof}

\begin{exmp}\label{ex:crossed-converses}
The converses of Theorem~\ref{relring}(v) and (vi) do 
not hold for general associative unital coefficient rings.
Indeed, let $K$ be a field, let $V$ be a vector space 
over $K$ with a countably infinite basis
$v_1,v_2,\ldots$, and put $R:=\End_K(V)$.
Define $a_1,a_2,b_1,b_2\in R$ by
\[
b_1(v_m):=v_{2m-1},
\qquad
b_2(v_m):=v_{2m}
\]
for every $m\geq 1$, and
\[
a_1(v_{2m-1}):=v_m,
\qquad
a_1(v_{2m}):=0,
\]
\[
a_2(v_{2m}):=v_m,
\qquad
a_2(v_{2m-1}):=0.
\]
Then $a_ib_j=\delta_{i,j}1_R$ for all $i,j\in\{1,2\}$, and
$b_1 a_1 + b_2 a_2 = 1_R$. Let
$G := \langle g\mid g^2=e\rangle$ and let 
$\rho:=\overline{3}^{\,2}$. Define a $G$-grading on $\rho$ by
\[
\rho_e := \{(1,1),(1,2),(2,1),(2,2),(3,3)\}
\quad \mbox{and} \quad
\rho_g := \{(1,3),(2,3),(3,1),(3,2)\}.
\]
A direct computation gives
$\rho_e\rho_e=\rho_e$, $\rho_e\rho_g=\rho_g=\rho_g\rho_e$,
$\rho_g\rho_g=\rho_e$.
Thus the grading on $\rho$ is strong.
We claim that it is not an epsilon-crossed product. 
By Remark~\ref{rem:epsilon-set},
$E_g=\rho_g\rho_{g^{-1}}\cap\Delta
=\rho_e\cap\Delta
=\Delta$.
Suppose that there are subsets
$U_g, V_g\subseteq\rho_g$ such that $U_gV_g=\Delta$.
Since $(1,1)\in U_gV_g$, we must have
$(1,3)\in U_g$ and $(3,1)\in V_g$. 
Similarly, since $(2,2)\in U_gV_g$, we must have
$(2,3)\in U_g$ and $(3,2)\in V_g$. It follows that
$(1,2)=(1,3)(3,2)\in U_gV_g$,
contradicting $U_gV_g=\Delta$. Hence the grading on 
$\rho$ is not an epsilon-crossed product and, 
in particular, is not a crossed product.

Now equip $S:=M_3(R)$ with the very good grading 
induced by the above grading on $\rho$.
Consider the homogeneous element
$u_g := a_1e_{1,3} + a_2e_{2,3} + b_1e_{3,1}
+ b_2e_{3,2} \in S_g$.
Using the relations satisfied by $a_1,a_2,b_1,b_2$, we obtain
\[
u_g^2 = \sum_{i,j=1}^2 a_ib_j e_{i,j} +
(b_1a_1+b_2a_2)e_{3,3} = e_{1,1}+e_{2,2}+e_{3,3} = 1_\Delta.
\]
Thus $u_g$ is invertible, with $u_g^{-1}=u_g$. 
Since $1_\Delta\in S_e$ is invertible, the induced 
grading on $S$ is a crossed product and 
therefore also an epsilon-crossed product.
\end{exmp}

\begin{cor}\label{cor:epsilon-criterion}
Suppose that $M_n(\rho,R)$ is equipped with a very good
$G$-grading. Then the following assertions are equivalent:
\begin{enumerate}[{\rm (i)}]

\item $M_n(\rho,R)$ is epsilon-strongly $G$-graded;

\item for every $g\in G$ and every $(i,j)\in\rho_g$,
we have $(i,i)\in\rho_g\rho_{g^{-1}}$ and 
$(j,j)\in\rho_{g^{-1}}\rho_g$.

\end{enumerate}
In that case, for each $g \in G$, we have 
$\epsilon_g = \sum_{i \in \overline{n},
(i,i)\in\rho_g\rho_{g^{-1}}} e_{i,i}.$
\end{cor}

\begin{proof}
By Theorem~\ref{relring}, the grading on $M_n(\rho,R)$ is
epsilon-strong if and only if the grading on $\rho$ is
epsilon-strong. 

Suppose first that the grading is epsilon-strong. By
Remark~\ref{rem:epsilon-set},
$E_g=\rho_g\rho_{g^{-1}}\cap\Delta$.
If $(i,j)\in\rho_g$, then
$E_g(i,j)=(i,j)=(i,j)E_{g^{-1}}$, and 
hence
$(i,i)\in\rho_g\rho_{g^{-1}}$
and $(j,j)\in\rho_{g^{-1}}\rho_g$.
Thus (ii) holds.

Conversely, suppose that (ii) holds. For every $g\in G$, 
put $E_g := \bigl(\rho_g\rho_{g^{-1}}\bigr)\cap\Delta$.
If $(i,j)\in\rho_g$, then (ii) gives
$(i,i)\in E_g$ and $(j,j)\in E_{g^{-1}}$.
Hence $E_g(i,j)=(i,j)=(i,j)E_{g^{-1}}$,
so the grading on $\rho$ is epsilon-strong. 
Theorem~\ref{relring} therefore shows that the 
grading on $M_n(\rho,R)$ is epsilon-strong.

Finally, put $\epsilon_g :=
\sum_{i\in\overline{n}, (i,i)\in\rho_g\rho_{g^{-1}}} e_{i,i}$.
Since $M_n(\rho,R)_gM_n(\rho,R)_{g^{-1}} =
\bigoplus_{(i,j)\in\rho_g\rho_{g^{-1}}}Re_{i,j}$,
we have $\epsilon_g\in M_n(\rho,R)_gM_n(\rho,R)_{g^{-1}}$.
Assertion (ii) implies that
$\epsilon_gs=s=s\epsilon_{g^{-1}}$
for every $s\in M_n(\rho,R)_g$. 
Thus $\epsilon_g$ is the sought unique epsilon element satisfying \eqref{eqep}. 
\end{proof}

\begin{cor}\label{cor:equivalence-epsilon-ring}
Suppose that $\rho$ is an equivalence relation. Then every very good
$G$-grading on  $M_n(\rho,R)$ is epsilon-strong.
\end{cor}

\begin{proof}
Let $(\rho_g)_{g\in G}$ be the $G$-grading on $\rho$ inducing the
very good grading on $M_n(\rho,R)$. Let $g\in G$ and
$(i,j)\in\rho_g$. Since $\rho$ is an equivalence relation, we have
$(j,i)\in\rho$. Lemma~\ref{lem:TransposeInverse} therefore 
gives $(j,i)\in\rho_{g^{-1}}$. Consequently,
$(i,i)=(i,j)(j,i)\in\rho_g\rho_{g^{-1}}$
and $(j,j)=(j,i)(i,j)\in\rho_{g^{-1}}\rho_g$.
The result now follows from
Corollary~\ref{cor:epsilon-criterion}.
\end{proof}

The preceding criterion also applies to preorders that are not
equivalence relations.

\begin{exmp}\label{enoeq}
Consider the preorder
\[
\rho
:=
\{
(1,1),(2,2),(3,3),(4,4),(1,2),(3,4),(4,3)
\}
\]
on $\overline{4}$. The relation $\rho$ is not an 
equivalence relation, since
$(1,2)\in\rho$ but $(2,1)\notin\rho$.
Let $\Z_3:=\{0,1,2\}$, and define a $\Z_3$-grading on 
$\rho$ by
\[
\rho_0
:=
\{(1,1),(2,2),(3,3),(4,4),(1,2)\},
\quad
\rho_1:=\{(3,4)\},
\quad
\rho_2:=\{(4,3)\}.
\]
By Proposition~\ref{prop:rhodefinesS}, this grading 
induces a very good $\Z_3$-grading on
$S:=M_4(\rho,R)$. We have $\rho_0\rho_0=\rho_0$,
$\rho_1\rho_2=\{(3,3)\}$ and 
$\rho_2\rho_1=\{(4,4)\}$.
Consequently, the condition in
Corollary~\ref{cor:epsilon-criterion}(ii) holds for 
every homogeneous
relation $\rho_g$. Hence the induced grading on $S$ is
epsilon-strong. Its epsilon elements are
$\epsilon_0=1_\Delta$, $\epsilon_1=e_{3,3}$ and
$\epsilon_2=e_{4,4}$.
\end{exmp}

Next, we investigate strongly graded structural matrix rings further.

\begin{theorem}\label{thm:strong-cardinality}
Suppose that $M_n(\rho,R)$ is 
equipped with a very good strong $G$-grading. 
Then
$|G|\leq n$. Moreover, if $|G|=n$, then
$\rho=\overline{n}^{\,2}$, and hence $M_n(\rho,R)=M_n(R)$.
\end{theorem}

\begin{proof}
By Theorem~\ref{relring}, the $G$-grading on 
$\rho$ is strong. For each $g\in G$, put
\[
J_g := \bigl\{ j\in\overline{n} \bigm|
(1,j)\in\rho_g \text{ and } (j,1)\in\rho_{g^{-1}}
\bigr\}.
\]
Since the grading is strong, $\rho_g\rho_{g^{-1}}=\rho_e$.
Moreover, $(1,1)\in\Delta\subseteq\rho_e$. Hence
$(1,1)\in\rho_g\rho_{g^{-1}}$, 
so there is some $j\in\overline{n}$ such that
$(1,j)\in\rho_g$ and $(j,1)\in\rho_{g^{-1}}$.
Thus $J_g\neq\emptyset$ for every $g\in G$.
We claim that the sets $J_g$, $g\in G$, are pairwise 
disjoint. Indeed, if $j \in J_g\cap J_h$, then
$(1,j)\in\rho_g\cap\rho_h$.
Since the homogeneous components of $\rho$ are 
pairwise disjoint, it follows that $g=h$. 
Thus $(J_g)_{g\in G}$ is a family of pairwise
disjoint nonempty subsets of the $n$-element set 
$\overline{n}$. Consequently, $|G|\leq n$.
Suppose now that $|G|=n$. Since the sets $J_g$, 
$g\in G$, are pairwise disjoint and nonempty, 
their union is $\overline{n}$. Hence,
for every $j\in\overline{n}$, there is some $g\in G$ 
such that $j \in J_g$. In particular,
$(1,j)\in\rho$ and $(j,1)\in\rho$.
Therefore, for arbitrary $i,j\in\overline{n}$, 
transitivity of $\rho$ gives
$(i,j)=(i,1)(1,j)\in\rho$. Thus
$\rho=\overline{n}^{\,2}$, and consequently
$M_n(\rho,R)=M_n(R)$.
\end{proof}

\section{Very good gradings and partial actions}\label{partial}

A very good $G$-grading on $M_n(\rho,R)$ is determined by a $G$-grading $(\rho_g)_{g\in G}$ on the preorder $\rho$. In this section, we consider the case where $\rho$ is an equivalence relation and the neutral component of the induced ring grading consists precisely of the diagonal matrices. We show that, under these assumptions, each relation $\rho_g$ is the graph of a partial bijection and that these partial bijections form a free partial action of $G$ on $\n$ whose orbit equivalence relation is $\rho$. Conversely, every such partial action induces a very good $G$-grading on $M_n(\rho,R)$. This yields a bijective correspondence between the two types of structures.

\begin{defn}\label{def:partial-action}
A \emph{partial action} of $G$ on a set $X$ is a family
$\alpha=(X_g,\alpha_g)_{g\in G}$,
where $X_g\subseteq X$ and
$\alpha_g:X_{g^{-1}}\to X_g$ is a bijection for every $g\in G$,
such that the following assertions hold:
\begin{itemize}
\item[(P1)] $X_e=X$ and $\alpha_e=\id_X$;
\item[(P2)] if $g,h\in G$, $x\in X_{h^{-1}}$, and
$\alpha_h(x)\in X_{g^{-1}}$, then $x\in X_{(gh)^{-1}}$ and
$\alpha_{gh}(x)=\alpha_g(\alpha_h(x))$.
\end{itemize}
The partial action is called \emph{global} if $X_g=X$ for every
$g\in G$. It is called \emph{free} if, whenever
$x\in X_{g^{-1}}\cap X_{h^{-1}}$ and
$\alpha_g(x)=\alpha_h(x)$, one has $g=h$.
The \emph{orbit equivalence relation} $\sim_\alpha$ 
induced by $\alpha$ is defined by
$i\sim_\alpha j \Leftrightarrow \alpha_g(j)=i$
for some $g\in G$ with $j\in X_{g^{-1}}$.
The partial action is called \emph{transitive} if
$i\sim_\alpha j$ for all $i,j\in X$.
\end{defn}

The orbit relation induced by any partial action is 
an equivalence relation. Thus, if a $G$-grading on a 
preorder $\rho$ is to be described by a partial action 
whose orbit relation is $\rho$, it is necessary to 
assume that $\rho$ is an equivalence relation. For the
remainder of this section, we therefore assume that 
$\rho$ is an equivalence relation on $\n$. 
We put $S:=M_n(\rho,R)$ and $D:=\bigoplus_{i=1}^n Re_{i,i}$.

\begin{defn}
We denote by $\mathcal{G}(\rho,G,R)$ the set of all very good
$G$-gradings $\mathcal{S}=(S_g)_{g\in G}$
on $S$ such that $S_e=D$. We also denote by
$\mathcal{A}(\rho,G)$ the set of all free partial actions 
of $G$ on
$\n$ whose orbit equivalence relation is $\rho$.
\end{defn}

\begin{defn}\label{def:partial-action-from-grading}
Let $\mathcal{S}=(S_g)_{g\in G}
\in\mathcal{G}(\rho,G,R)$. Take $g \in G$. Put
\[
\rho_g^{\mathcal{S}} :=
\{(i,j)\in\rho\mid e_{i,j}\in S_g\}, \quad \mbox{and} \quad
X_g^{\mathcal{S}} :=
\{i \in \n \mid (i,j)\in\rho_g^{\mathcal{S}} \
\mbox{for some} \ j \in \n \}.
\]
Consider the rule
$\alpha_g^{\mathcal{S}}(j)=i \Leftrightarrow
(i,j)\in\rho_g^{\mathcal{S}}$, where
$j\in X_{g^{-1}}^{\mathcal{S}}$ and 
$i\in X_g^{\mathcal{S}}$.
We then write $\alpha^{\mathcal{S}} := 
\bigl( X_g^{\mathcal{S}}, \alpha_g^{\mathcal{S}}
\bigr)_{g\in G}$ and define
$\Phi(\mathcal{S}):=\alpha^{\mathcal{S}}$.
\end{defn}

\begin{lemma}\label{lem:grading-to-partial-action}
The map $\Phi:
\mathcal{G}(\rho,G,R) \to \mathcal{A}(\rho,G)$
is well defined.
\end{lemma}

\begin{proof}
Let $\mathcal{S}=(S_g)_{g\in G}
\in\mathcal{G}(\rho,G,R)$.
Since $\mathcal{S}$ is a very good grading, the family
$\bigl(\rho_g^{\mathcal{S}}\bigr)_{g\in G}$
is a $G$-grading on $\rho$. Moreover, the equality $S_e=D$ implies that $\rho_e^{\mathcal{S}}=\Delta$.
Since $\rho$ is an equivalence relation,
Proposition~\ref{prop:i-ii-iii} gives
$\bigl(\rho_g^{\mathcal{S}}\bigr)^{-1} =
\rho_{g^{-1}}^{\mathcal{S}}$
for every $g\in G$.
We first show that $\alpha_g^{\mathcal{S}}:
X_{g^{-1}}^{\mathcal{S}} \to X_g^{\mathcal{S}}$ is a well-defined bijection.
Let $j\in X_{g^{-1}}^{\mathcal{S}}$. By the definition of
$X_{g^{-1}}^{\mathcal{S}}$, there is some $i\in\n$ with
$(j,i)\in\rho_{g^{-1}}^{\mathcal{S}}$.
It follows from $\bigl(\rho_g^{\mathcal{S}}\bigr)^{-1} =
\rho_{g^{-1}}^{\mathcal{S}}$ that
$(i,j)\in\rho_g^{\mathcal{S}}$.
In particular, $i\in X_g^{\mathcal{S}}$.
This element $i$ is unique. Indeed, suppose that
$(i,j),(i',j)\in\rho_g^{\mathcal{S}}$.
By $\bigl(\rho_g^{\mathcal{S}}\bigr)^{-1} =
\rho_{g^{-1}}^{\mathcal{S}}$, we have
$(j,i')\in\rho_{g^{-1}}^{\mathcal{S}}$.
Therefore, $(i,i') = (i,j)(j,i') \in
\rho_g^{\mathcal{S}}\rho_{g^{-1}}^{\mathcal{S}}
\subseteq \rho_e^{\mathcal{S}} = \Delta$.
Consequently, $i=i'$, and hence
$\alpha_g^{\mathcal{S}}$ is well defined.
The relation $\bigl(\rho_g^{\mathcal{S}}\bigr)^{-1} =
\rho_{g^{-1}}^{\mathcal{S}}$ also shows that
$\alpha_{g^{-1}}^{\mathcal{S}}$ is the inverse of
$\alpha_g^{\mathcal{S}}$. Thus
$\alpha_g^{\mathcal{S}}:
X_{g^{-1}}^{\mathcal{S}} \to X_g^{\mathcal{S}}$
is a bijection.

Since $\rho_e^{\mathcal{S}}=\Delta$, we have
$X_e^{\mathcal{S}}=\n$ and $\alpha_e^{\mathcal{S}}=\id_{\n}$.

Now let $g,h\in G$, and suppose that
$j \in X_{h^{-1}}^{\mathcal{S}}$ and 
$\alpha_h^{\mathcal{S}}(j) \in X_{g^{-1}}^{\mathcal{S}}$.
Put $l:=\alpha_h^{\mathcal{S}}(j)$ and 
$i:=\alpha_g^{\mathcal{S}}(l)$. Then
$(i,l)\in\rho_g^{\mathcal{S}}$ and 
$(l,j)\in\rho_h^{\mathcal{S}}$.
It follows that $(i,j) \in \rho_g^{\mathcal{S}}
\rho_h^{\mathcal{S}} \subseteq \rho_{gh}^{\mathcal{S}}$.
Consequently, $\alpha_{gh}^{\mathcal{S}}(j) = i =
\alpha_g^{\mathcal{S}} \bigl(\alpha_h^{\mathcal{S}}(j)\bigr)$.
Thus $\alpha^{\mathcal{S}}$ is a partial action of $G$ on $\n$.

We now show that this partial action is free. Let
$j\in X_{g^{-1}}^{\mathcal{S}} \cap X_{h^{-1}}^{\mathcal{S}}$
and $\alpha_g^{\mathcal{S}}(j) = \alpha_h^{\mathcal{S}}(j)
= i$. Then $(i,j)\in \rho_g^{\mathcal{S}} \cap
\rho_h^{\mathcal{S}}$.
Since the sets $\rho_g^{\mathcal{S}}$, $g\in G$, are pairwise
disjoint, it follows that $g=h$.

Finally, for $i,j\in\n$, we have
$(i,j) \in \rho \Leftrightarrow
(i,j)\in\rho_g^{\mathcal{S}}$ for some $g\in G$
$\Leftrightarrow \alpha_g^{\mathcal{S}}(j)=i$
for some $g\in G$.
Consequently, the orbit equivalence relation of
$\alpha^{\mathcal{S}}$ is $\rho$. Thus
$\alpha^{\mathcal{S}}
\in
\mathcal{A}(\rho,G)$,
and hence $\Phi$ is well defined.
\end{proof}

\begin{defn}\label{def:grading-from-partial-action}
Let $\alpha=(X_g,\alpha_g)_{g\in G} \in\mathcal{A}(\rho,G)$.
For every $g\in G$, put 
\[
\rho_g^\alpha :=
\{(\alpha_g(j),j)\mid j\in X_{g^{-1}}\} 
\quad \mbox{and} \quad
S_g^\alpha := \bigoplus_{j\in X_{g^{-1}}} Re_{\alpha_g(j),j}.
\]
We write $\mathcal{S}^\alpha := (S_g^\alpha)_{g\in G}$ 
and define $\Psi(\alpha):=\mathcal{S}^{\alpha}$.
\end{defn}

\begin{lemma}\label{lem:partial-action-to-grading}
The map
$\Psi: \mathcal{A}(\rho,G) \to \mathcal{G}(\rho,G,R)$
is well defined.
\end{lemma}

\begin{proof}
Let $\alpha=(X_g,\alpha_g)_{g\in G} \in \mathcal{A}(\rho,G)$.
We first show that $(\rho_g^\alpha)_{g\in G}$
is a $G$-grading on $\rho$. Since the orbit 
equivalence relation of $\alpha$ is $\rho$, we have
$\rho = \bigcup_{g\in G}\rho_g^\alpha$. Let
$(i,j)\in\rho_g^\alpha\cap\rho_h^\alpha$.
Then $\alpha_g(j)=i=\alpha_h(j)$.
Since $\alpha$ is free, it follows that $g=h$. Thus the sets
$\rho_g^\alpha$, $g\in G$, are pairwise disjoint.
It remains to verify compatibility with composition. 
Suppose that $(i,l)\in\rho_g^\alpha$ and 
$(l,j)\in\rho_h^\alpha$. Then $i=\alpha_g(l)$ 
and $l=\alpha_h(j).$
In particular, $\alpha_g(\alpha_h(j))$ is defined. By the
partial action identity, $i = \alpha_g(\alpha_h(j)) =
\alpha_{gh}(j)$.
Hence $(i,j)\in\rho_{gh}^\alpha$.
Therefore, $\rho_g^\alpha\rho_h^\alpha
\subseteq \rho_{gh}^\alpha$. Consequently,
$(\rho_g^\alpha)_{g\in G}$ is a $G$-grading on $\rho$.
By Proposition~\ref{prop:rhodefinesS}, this $G$-grading on $\rho$ induces the very good $G$-grading
$(S_g^\alpha)_{g\in G}$ on $S$
from Definition \ref{def:grading-from-partial-action}.
Finally, since $X_e=\n$ and $\alpha_e=\id_{\n}$, we have
$\rho_e^\alpha = \{(i,i)\mid i\in\n\} = \Delta$.
It follows that $S_e^\alpha = \bigoplus_{i=1}^n Re_{i,i}
= D$. Thus $\mathcal{S}^\alpha \in \mathcal{G}(\rho,G,R)$,
and hence $\Psi$ is well defined.
\end{proof}

\begin{theorem}\label{thm:partial-action-correspondence}
The maps $\Phi$ and $\Psi$ are mutually inverse.
\end{theorem}

\begin{proof}
Let $\mathcal{S}=(S_g)_{g\in G} \in \mathcal{G}(\rho,G,R)$,
and let $(\rho_g^{\mathcal{S}})_{g\in G}$
be the associated $G$-grading on $\rho$.
By the definitions of $\Phi$ and $\Psi$, the 
degree-$g$ relation
associated to $\Psi(\Phi(\mathcal{S}))$ is
\[
\bigl\{(\alpha_g^{\mathcal{S}}(j),j)
\bigm|
j\in X_{g^{-1}}^{\mathcal{S}}
\bigr\} = \rho_g^{\mathcal{S}}.
\]
Thus the degree-$g$ component of
$\Psi(\Phi(\mathcal{S}))$ is $S_g$ for every $g\in G$. 
Hence, $\Psi\circ\Phi = \id_{\mathcal{G}(\rho,G,R)}$.

Conversely, let $\alpha=(X_g,\alpha_g)_{g\in G}
\in\mathcal{A}(\rho,G)$.
The degree-$g$ relation associated to $\Psi(\alpha)$ is
$\rho_g^\alpha = \{(\alpha_g(j),j)\mid j\in X_{g^{-1}}\}$.
The set of first coordinates occurring in $\rho_g^\alpha$ is
$\alpha_g(X_{g^{-1}}) = X_g$.
Moreover, the partial bijection recovered from
$\rho_g^\alpha$ is the map $j\mapsto\alpha_g(j)$,
$j \in X_{g^{-1}}$.
Thus the construction in
Lemma~\ref{lem:grading-to-partial-action} recovers both 
$X_g$ and $\alpha_g$ for every $g\in G$. Consequently,
$\Phi(\Psi(\alpha)) = \alpha$. Hence
$\Phi\circ\Psi = \id_{\mathcal{A}(\rho,G)}$.
Therefore, $\Phi$ and $\Psi$ are mutually inverse.
\end{proof}

\begin{cor}
Let $G$ be a finite group, and let $C_1,\ldots,C_k$ be the equivalence classes of $\rho$. 
The following two assertions are equivalent:
\begin{enumerate}[{\rm (i)}]
	\item $\mathcal{A}(\rho,G)$ is nonempty;
	\item $|C_i|\leq |G|$ for every $i \in \{1,\ldots,k\}$.
\end{enumerate}
Whenever these two equivalent assertions hold true, we have
\begin{displaymath}
		|\mathcal{A}(\rho,G)| =
		\prod_{i=1}^{k}
		\frac{(|G|-1)!}{(|G|-|C_i|)!}.
\end{displaymath}
\end{cor}

\begin{proof}
For every $i \in \{1,\ldots,k\}$, choose some $m_i\in C_i$.
By Theorem~\ref{thm:partial-action-correspondence}, it suffices to
count the maps $f\colon \rho\to G$ satisfying
$f(x,y)f(y,z)=f(x,z)$ for $(x,y),(y,z) \in \rho$,
and
$f^{-1}(e)=\Delta.$

For $i \in \{1,\ldots,k\}$,
define $a_i : C_i \to G$ by $a_i(x):=f(m_i,x)$ for $x\in C_i$.
Then, for all $i\in \{1,\ldots,k\}$ and $x,y\in C_i$, we have $a_i(m_i)=e$ and
$f(x,y)=a_i(x)^{-1}a_i(y)$.
Consequently, $f^{-1}(e)=\Delta$ if and only if every $a_i$ is
injective.

Conversely, given injections $a_i : C_i \hookrightarrow G$ satisfying $a_i(m_i)=e$
for every $i$, define $f(x,y):=a_i(x)^{-1}a_i(y)$ whenever $x,y \in C_i$.
Injectivity gives $f(x,y)=e$ if and only if $x=y$. 
Hence there is a bijection between
$\mathcal{A}(\rho,G)$
and 
$\prod_{i=1}^{k}
\bigl\{a_i\colon C_i\hookrightarrow G\mid a_i(m_i)=e\bigr\}.$

For each $i$, the set $\{a_i\colon C_i\hookrightarrow G\mid a_i(m_i)=e\bigr\}$ is nonempty exactly when $|C_i| \leq |G|$.

Suppose that $\mathcal{A}(\rho,G)$ is nonempty.
The number of choices for $a_i$ is
\begin{displaymath}
\frac{(|G|-1)!}{(|G|-|C_i|)!}.	
\end{displaymath}
Multiplying over the equivalence classes proves the desired equality.
\end{proof}

The equivalence relation hypothesis in 
Theorem~\ref{thm:partial-action-correspondence} is 
essential. If $\rho$ is merely a preorder, a 
homogeneous component of a very good grading need not 
be the graph of a partial function, even when the 
neutral component consists precisely of the diagonal 
matrices, as the following example shows.

\begin{exmp}\label{ex:non-equivalence-no-partial-action}
Let $G:=\mathbb Z_2=\{e,g\}$, and consider the preorder
$\rho := \Delta\cup\{(1,3),(2,3)\}$
on $\overline{3}$. Define a $G$-grading on $\rho$ by
$\rho_e:=\Delta$ and $\rho_g:=\{(1,3),(2,3)\}$.
This grading induces a very good $G$-grading on
$M_3(\rho,R)$ whose neutral component is
$D=\bigoplus_{i=1}^3 Re_{i,i}$.
However, $\rho_g$ is not the graph of a partial function on
$\overline{3}$. Under the convention used 
above, 
the relation
$\rho_g$ would require simultaneously
$\alpha_g(3)=1$ and $\alpha_g(3)=2$.
Consequently, this very good grading cannot be 
encoded by a partial
action of $G$ on $\overline{3}$ in the manner described 
above.
\end{exmp}

\begin{cor}\label{cor:partial-epsilon-crossed}
Suppose that $\rho$ is an equivalence relation on $\n$ and that
$(S_g)_{g\in G}$
is a very good $G$-grading
on $M_n(\rho,R)$ with $S_e=D=\bigoplus_{i=1}^n Re_{i,i}$.
Then the G-graded ring $M_n(\rho,R)$ is an epsilon-crossed product.
\end{cor}

\begin{proof}
Let $\mathcal{S}=(S_g)_{g\in G} \in \mathcal{G}(\rho,G,R)$,
and let $\alpha=(X_g,\alpha_g)_{g\in G}$
be the corresponding free partial action under
Theorem~\ref{thm:partial-action-correspondence}. 
For every $g\in G$, put 
$u_g := \sum_{j\in X_{g^{-1}}} e_{\alpha_g(j),j} \in S_g$
and $v_{g^{-1}} :=
\sum_{i\in X_g} e_{\alpha_{g^{-1}}(i),i} \in S_{g^{-1}}$.
Note that, by Corollaries~\ref{cor:epsilon-criterion}--\ref{cor:equivalence-epsilon-ring},
$\epsilon_g = \sum_{i \in \overline{n},
(i,i)\in\rho_g\rho_{g^{-1}}} e_{i,i}.$
Hence $u_gv_{g^{-1}} = \sum_{i\in X_g}e_{i,i} = \epsilon_g$
and $v_{g^{-1}}u_g = \sum_{j\in X_{g^{-1}}}e_{j,j} =
\epsilon_{g^{-1}}$.
Thus $\mathcal{S}$ is an epsilon-crossed product.
\end{proof}

We next specialize Theorem~\ref{thm:partial-action-correspondence} to
full matrix rings. Recall that a very good $G$-grading on
$M_n(\rho,R)$ is called \emph{elementary} if it is induced by an
$n$-tuple $(g_1,\ldots,g_n)\in G^n$ in the sense that
$\deg(e_{i,j})=g_ig_j^{-1}$ whenever $(i,j) \in \rho$.
It is called \emph{strictly elementary} if the elements
$g_1,\ldots,g_n$ are pairwise distinct.

\begin{cor}\label{cor:strict-elementary-partial}
Under the correspondence in
Theorem~\ref{thm:partial-action-correspondence}, the strictly
elementary $G$-gradings on $M_n(R)$ correspond precisely to the free
and transitive partial actions of $G$ on $\n$.
\end{cor}

\begin{proof}
Let $M_n(R)=\bigoplus_{g\in G}S_g$
be a strictly elementary grading induced by
$(g_1,\ldots,g_n)$. Then
$\deg(e_{i,j})=e \Leftrightarrow g_i=g_j
\Leftrightarrow i=j$. Hence $S_e=D$.
Since the underlying equivalence relation of $M_n(R)$ is
$\n\times\n$, its corresponding partial action has only 
one orbit and is therefore transitive.

Conversely, suppose that a very good grading on $M_n(R)$ 
satisfies $S_e=D$. Put $g_i:=\deg(e_{i,1})$, $i \in \n$.
Since $e_{i,j}=e_{i,1}e_{1,j}$ and 
$\deg(e_{1,j})=g_j^{-1}$, we obtain 
$\deg(e_{i,j})=g_ig_j^{-1}$.
If $g_i=g_j$, then $e_{i,j}\in S_e=D$, and hence $i=j$.
Thus the grading is strictly elementary.

The result now follows from
Theorem~\ref{thm:partial-action-correspondence}.
\end{proof}

\begin{remark}\label{rem:several-orbits}
For an equivalence relation with more than one equivalence class, the
condition $S_e=D$ need not imply that the grading can be induced by a
single injective $n$-tuple in $G$. Thus, in
Theorem~\ref{thm:partial-action-correspondence}, the natural general
condition is $S_e=D$, whereas strict elementarity is recovered in the
full matrix case by Corollary~\ref{cor:strict-elementary-partial}.
\end{remark}

The strong-grading condition has a particularly simple interpretation.

\begin{cor}\label{cor:strong-global-partial}
Under the correspondence in
Theorem~\ref{thm:partial-action-correspondence}, the $G$-grading on
$M_n(\rho,R)$ is strong 
if and only if the corresponding
partial action is global.
\end{cor}

\begin{proof}
Let $\alpha=(X_g,\alpha_g)_{g\in G}$
be the partial action corresponding to the very good grading
$\mathcal{S}=(S_g)_{g\in G}$,
and let $(\rho_g)_{g\in G}$ be the associated $G$-grading on
$\rho$. Since $\rho$ is an equivalence relation,
Corollary~\ref{cor:equivalence-epsilon-ring} shows that
$\mathcal{S}$ is epsilon-strongly $G$-graded.

As observed in the proof of Corollary~\ref{cor:partial-epsilon-crossed},
we have $\epsilon_g = \sum_{i \in X_g} e_{ii}$ for each $g\in G$.
Therefore, an epsilon-strong grading is strong 
$\Leftrightarrow$ $\epsilon_g=1_\Delta$ for each $g\in G$ $\Leftrightarrow$
$X_g = \n$ for each $g\in G$ $\Leftrightarrow$ 
the partial action $\alpha$ is global.
\end{proof}

\begin{cor}\label{cor:regular-action}
Under the correspondence in
Theorem~\ref{thm:partial-action-correspondence}, a strictly
elementary $G$-grading on $M_n(R)$ is strong if and only
if the corresponding partial action is global.
Consequently, $M_n(R)$ admits a strong strictly
elementary $G$-grading if and only if $|G|=n$. In that case, the
corresponding action is equivalent to the regular action of $G$ on
itself.
\end{cor}

\begin{proof}
By Corollary~\ref{cor:strict-elementary-partial}, a partial action
corresponding to a strictly elementary grading on $M_n(R)$ is free and
transitive. By Corollary~\ref{cor:strong-global-partial}, a $G$-grading
is strong if and only if this partial action is global.
A free and transitive global action of $G$ on $\n$ is regular, and
therefore $|G|=n$. Conversely, if $|G|=n$, then, after identifying
$\n$ with $G$, the regular action of $G$ on itself corresponds to a
strong strictly elementary $G$-grading on $M_n(R)$.
\end{proof}

\begin{exmp}\label{ex:C6-partial}
Write $C_6:=\mathbb Z/6\mathbb Z$ additively, and put
$X:=\{1,2,3,4\}\subseteq C_6$. Restrict the translation action of
$C_6$ on itself to $X$. More precisely, for $g\in C_6$, put
$X_g:=X\cap(g+X)$ and define
$\alpha_g : X_{-g} \to X_g$ by $\alpha_g(x):=g+x$,
for $x \in X_{-g}$. This is a free and transitive 
partial action of $C_6$ on $X$. Under
the identification $X=\overline{4}$, the corresponding strictly
elementary grading on $S:=M_4(R)$ is induced by the tuple
$(1,2,3,4)$ and is given by $S_0 =\bigoplus_{i=1}^4Re_{i,i}$,
$ S_1 = Re_{2,1}\oplus Re_{3,2}\oplus Re_{4,3}$,
$S_2 = Re_{3,1}\oplus Re_{4,2}$,
$S_3 = Re_{1,4}\oplus Re_{4,1}$,
$S_4 = Re_{1,3}\oplus Re_{2,4}$,
$S_5 = Re_{1,2}\oplus Re_{2,3}\oplus Re_{3,4}$.
The partial action is not global, and hence the grading 
is not strong by Corollary~\ref{cor:strong-global-partial}. 
It is, however, an epsilon-crossed product by 
Corollary~\ref{cor:partial-epsilon-crossed}.
\end{exmp}

Let $S,S'$ be $G$-graded rings
with gradings $(S_g)_{g\in G}$ and $(S'_g)_{g\in G}$.
Recall that $S$ and $S'$, and their gradings, are said to be \emph{graded isomorphic}
if there is a ring isomorphism $\phi : S \to S'$
such that $\phi(S_g) = S'_g$ for every $g\in G$.
In that case, $\phi$ is referred to as a \emph{graded isomorphism}.

We finish this section by, in the case where $R$ is a field, classifying the gradings in
Theorem~\ref{thm:partial-action-correspondence} up to graded
isomorphism.

\begin{defn}\label{def:equivalent-partial-actions}
Two partial actions
$\alpha=(X_g,\alpha_g)_{g\in G}$ and
$\beta=(Y_g,\beta_g)_{g\in G}$ of $G$ on sets $X$ and $Y$,
respectively, are called \emph{equivalent} if there is a bijection $\sigma : X \to Y$ with $\sigma(X_g) = Y_g$
and $\sigma(\alpha_g(x)) = \beta_g(\sigma(x))$
for all $g \in G$ and $x \in X_{g^{-1}}$.
\end{defn}

\begin{theorem}\label{thm:classification-partial-actions}
Let $K$ be a field, let $\rho$ be an equivalence relation 
on $\n$, and put $S:=M_n(\rho,K)$ and 
$D: = \bigoplus_{i=1}^n Ke_{i,i}$.
The bijection in
Theorem~\ref{thm:partial-action-correspondence} induces a bijection
between the set of graded $K$-algebra isomorphism classes of very good
$G$-gradings on $S$ with $S_e=D$, and the set of equivalence
classes of free partial actions of $G$ on $\n$ whose orbit equivalence
relation is $\rho$.
\end{theorem}

\begin{proof}
Let $\alpha=(X_g,\alpha_g)_{g\in G}$ and
$\beta=(Y_g,\beta_g)_{g\in G}$ be equivalent partial actions, and let
$\sigma:\n\to\n$ be an equivalence between them. Since both orbit
equivalence relations are $\rho$, the bijection $\sigma$ is an
automorphism of the relation $\rho$. Hence the $K$-linear 
map $\Phi_\sigma : S \to S$, defined by 
$\Phi_\sigma(e_{i,j}) := e_{\sigma(i),\sigma(j)}$,
is a $K$-algebra automorphism. If $(i,j)\in\rho_g$ 
for the grading associated to $\alpha$, then $i=\alpha_g(j)$, 
and therefore $\sigma(i) = \sigma(\alpha_g(j)) = 
\beta_g(\sigma(j))$.
Thus $(\sigma(i),\sigma(j))$ belongs to the 
degree-$g$ part of the grading associated to $\beta$.
Consequently, $\Phi_\sigma$ is a graded isomorphism.

Conversely, let $(S_g)_{g\in G}$ and 
$(S'_g)_{g\in G}$
be two very good $G$-gradings on $S$ satisfying $S_e=S'_e=D$, 
and suppose that $\Phi:(S,(S_g)_{g\in G}) \to
(S,(S'_g)_{g\in G})$
is a graded $K$-algebra isomorphism. Since $\Phi(D)=D$ and
$D\cong K^n$, the restriction of $\Phi$ to $D$ permutes 
the primitive idempotents of $D$. Thus there is a 
permutation $\sigma$ of $\n$ with
$\Phi(e_{i,i})=e_{\sigma(i),\sigma(i)}$ for every $i \in \n$.

Take $(i,j)\in\rho$. Then
$0 \neq \Phi(e_{i,j}) = 
\Phi(e_{i,i}e_{i,j}e_{j,j}) =
e_{\sigma(i),\sigma(i)}\Phi(e_{i,j})
e_{\sigma(j),\sigma(j)}$.
It follows that $(\sigma(i),\sigma(j))\in\rho$ and that
$\Phi(e_{i,j}) = \lambda_{i,j} e_{\sigma(i),\sigma(j)}$
for some $\lambda_{i,j}\in K \setminus \{0\}$. 
Applying the same argument to
$\Phi^{-1}$ shows that $\sigma$ is an automorphism of $\rho$.

Let $\alpha$ and $\beta$ be the partial actions corresponding to the
two gradings. If $j\in X_{g^{-1}}$ and
$\alpha_g(j)=i$, then $e_{i,j}\in S_g$. 
Since $\Phi$ is graded, $\Phi(e_{i,j})=\lambda_{i,j}
 e_{\sigma(i),\sigma(j)}$ gives
$e_{\sigma(i),\sigma(j)}\in S'_g$. Hence
$\beta_g(\sigma(j))=\sigma(i)=\sigma(\alpha_g(j))$.
It also follows that $\sigma(X_{g^{-1}})=Y_{g^{-1}}$ for every
$g\in G$. Thus $\sigma$ is an equivalence between $\alpha$ and
$\beta$.
\end{proof}

\begin{cor}\label{cor:classification-strict-elementary}
Let $K$ be a field. The correspondence in
Theorem~\ref{thm:partial-action-correspondence} induces a bijection
between the set of graded $K$-algebra isomorphism classes of strictly
elementary $G$-gradings on $M_n(K)$ and the set of equivalence classes
of free and transitive partial actions of $G$ on $\n$.
\end{cor}

\begin{proof}
This follows from Corollary~\ref{cor:strict-elementary-partial} and
Theorem~\ref{thm:classification-partial-actions}, applied to
$\rho=\n\times\n$.
\end{proof}

\end{document}